\documentclass[reqno]{amsart}
\usepackage[english]{babel}
\usepackage{amssymb,verbatim}
\usepackage[T1]{fontenc}
\newtheorem{theorem}{Theorem}[section]
\newtheorem{proposition}[theorem]{Proposition}
\newtheorem{lemma}[theorem]{Lemma}
\newtheorem{corollary}[theorem]{Corollary}
\theoremstyle{definition}
\newtheorem{definition}[theorem]{Definition}
\newtheorem{example}[theorem]{Example}
\theoremstyle{remark}
\newtheorem*{remark}{Remark}
\numberwithin{equation}{section}
\newcommand{\RE}{\mbox{$\mathbb{R}$}}
\newcommand{\C}{\mbox{$\mathbb{C}$}}

\begin{document}

\title{The classification of holomorphic $(m,n)$--subharmonic morphisms}

\author{Per \AA hag}\address{Department of Mathematics and Mathematical Statistics\\ Ume\aa \
University\\SE-901~87 Ume\aa, Sweden \\ Sweden}\email{Per.Ahag@math.umu.se}

\author{Rafa\l\ Czy{\.z}}\address{Faculty of Mathematics and Computer Science, Jagiellonian
University,
\L ojasiewicza~6, 30-348 Krak\'ow, Poland}
\email{Rafal.Czyz@im.uj.edu.pl}

\author{Lisa Hed}\address{Department of Mathematics and Mathematical Statistics\\ Ume\aa \
University\\SE-901~87 Ume\aa, Sweden \\ Sweden}\email{Lisa.Hed@math.umu.se}


\begin{abstract}
We study the problem of classifying the holomorphic $(m,n)$-$\!$ subharmonic morphisms in complex
space. This determines which holomorphic mappings preserves $m$-subharmonicity in the sense that
the composition of the holomorphic mapping with a $m$-subharmonic functions is $n$-subharmonic.
We show that there are three different scenarios depending on the underlying dimensions, and the
model itself. Either the holomorphic mappings are just the constant functions, or up to
composition with a homotethetic map, canonical orthogonal projections. Finally, there is a more
intriguing case when subharmonicity is gained in the sense of the Caffarelli-Nirenberg-Spruck
framework.
\end{abstract}

\maketitle

\section{Introduction}

 Let $X,Y$, and $K$ be suitable spaces (that shall be discussed later). A continuous function
 $f:X\to Y$ is said to be
 a \emph{harmonic morphism} if, for every open set $V\subseteq Y$ with $f^{-1}(V)\neq \emptyset$,
 and for every harmonic function $h:V\to K$, the function $(h\circ f):f^{-1}(V)\to K$ is a harmonic
 function. In 1965, Constantinescu and Cornea~\cite{ConstantinescuCornea} made the first general
 study of harmonic morphisms in the general potential-theoretical  setting of continuous maps
 between harmonic spaces in the sense of Brelot (see also~\cite{Laine,Meghea}). Even so, the idea
 of harmonic morphisms is  more than a century older than the publication
 of~\cite{ConstantinescuCornea}. An early encounter of the idea of harmonic morphism is from 1848;
 an article written by Charles Gustave Jacob Jacobi~\cite{jacobi}. He considered the case when
 $X=\RE^3$, $Y=\C$, $K=\C$, and the harmonic functions $h:V\to\C$ were holomorphic functions. Many
 consider this the start of the subject of harmonic morphisms (see
 e.g.~\cite{BairdWood,Fuglede2000}). In the late seventies Fuglede~\cite{Fuglede1978} and
 Ishihara~\cite{Ishihara}, independently, characterized the harmonic morphisms in the case of
 Riemannian manifolds, i.e., $X=M$, $Y=N$ are Riemannian manifolds and $K=\RE$. Ever since the
 publication of those articles the subject of harmonic morphisms, and its generalization, have been
 considerably studied and used. A strong indication  of this is the mighty Bibliography of Harmonic
 Morphisms~\cite{GudmundssonBiblio} by Gudmundsson. We would like to mention the connection with
 Brownian motions, and stochastic processes~\cite{BanuelosOksendal, BernardCampbellDavie,
 CsinkOksendal1, CsinkOksendal2, CsinkFitzsimmonsOksendal,Darling, Levy}, the extension to
 nonlinear potential theory~\cite{HeinonenKilpelainenMartio}, and the consideration of
 pseudoharmonic morphisms~\cite{BarlettaDragomirUrakawa}. Furthermore, the applications to:
 potential theory~\cite{Fuglede2011}, minimal submanifolds~\cite{BairdGudmundsson}, and to physical
 gravity~\cite{mustafa}. For further information about harmonic morphisms, and its generalizations,
 we highly recommend the monograph~\cite{BairdWood} written by Baird and Wood.

In this article we shall consider complex spaces, i.e., $X=\C^N$, $Y=\C^M$, and our morphisms shall
be within the
Caffarelli-Nirenberg-Spruck model. But before we state our main result (Theorem~C), let us give a
thorough background on the motivation behind this paper.

 Assume that $\Omega$ and $\Omega'$ are two connected, and open subsets of $\mathbb C$, and let
 $f:\Omega\to\Omega'$ be a function. Then it is a classical result that the following conditions
 are equivalent:

\medskip

\begin{enumerate}\itemsep2mm

\item \emph{for every subharmonic function $\varphi:\Omega'\to\RE\cup\{-\infty\}$, the function
    $(\varphi\circ f):\Omega\to\RE\cup\{-\infty\}$ is subharmonic;}

\item \emph{for every harmonic function $h:\Omega'\to\RE$, the function $(h\circ f):\Omega\to\RE$
    is harmonic;}

\item \emph{$f$ is holomorphic or $\bar f$ is holomorphic.}

\end{enumerate}

\medskip

\noindent For a proof see e.g.~\cite{K1, K}. Thus, we can characterize those functions
$f:\Omega\to\Omega'$ such that condition (1), or (2), holds. Even though the work of
Jacobi~\cite{jacobi} is certainly an early encounter with harmonic morphisms as we know them today,
the geometrical idea behind the equivalence of (2) and (3) is due to Johann Carl Friedrich
Gau{\ss}~\cite{Gauss1a} from 1822 (published in~\cite{Gauss1b}; for an English translation
see~\cite{Gauss2, Gauss3}). In modern times, the above result was used for example in analytic
multi-valued function theory (see e.g.~\cite{K1} and the references therein).

 It is possible to generalize the above equivalences to higher dimensions. This will be discussed
 here in three directions. First, within several complex variables and pluripotential theory
 (Theorem~A), and then later within potential theory (Theorem~B). Finally, we shall prove a
 generalization to the Caffarelli-Nirenberg-Spruck model (Theorem~C). The first generalization
 states that:

\medskip

\noindent \textbf{Theorem~A.} \emph{Let $f=(f_1,\ldots,f_M):\Omega_N\to\Omega_M$ be a function,
where $\Omega_N\subseteq\mathbb{C}^N$ and $\Omega_M\subseteq\mathbb{C}^M$ are connected, and open
sets with $N,M\geq 1$. Then the following assertions are equivalent: }
\begin{enumerate}\itemsep2mm

\item \emph{for every plurisubharmonic function $\varphi:\Omega_M\to\RE\cup\{-\infty\}$, the
    function $(\varphi\circ f):\Omega_N\to\RE\cup\{-\infty\}$ is plurisubharmonic;}

\item \emph{for every pluriharmonic function $h:\Omega_M\to\RE$, the function $(h\circ
    f):\Omega_N\to\RE$ is pluriharmonic;}

\item \emph{$f$ is holomorphic or $\bar f$ is holomorphic.}

\end{enumerate}

\medskip

\noindent That an upper semicontinuous function $\varphi:\Omega_M\to\RE\cup\{-\infty\}$ is defined
to be a plurisubharmonic function if $(\varphi\circ f):\Omega_N\to\RE\cup\{-\infty\}$ is a
plurisubharmonic function, for every holomorphic mapping $f:\Omega_N\to\Omega_M$, goes back to the
legendary work of Oka~\cite{oka}, and Lelong~\cite{lelong}, in 1942. In fact, it is also common
knowledge that we only have to consider $\C$-linear isomorphisms $f:\Omega_N\to\Omega_M$ (see
e.g.~\cite{K}). It should
be noted that in condition (2) of Theorem~A we have to change harmonic functions to the
pluriharmonic functions, since they are a natural counterpart to harmonic functions in
pluripotential theory. We strongly believe Theorem~A is folklore among experts in several-variable
complex analysis, and pluripotential theory, but we could not find an elementary
proof in the literature. Therefore we have included one in Section~\ref{sec_firstpart}. For the
case of complex manifolds
see~\cite{Loubeau1,Loubeau2}.

Next, we look at Gudmundsson-Sigurdsson's  potential-theoretical generalization
from~\cite{gudmundssonsigurdsson}:

\medskip

\noindent \textbf{Theorem~B.} \emph{Let $f=(f_1,\ldots,f_M):\Omega_N\to\Omega_M$ be a holomorphic
mapping, where $\Omega_N\subseteq\mathbb{C}^N$, $N\geq 1$, and $\Omega_M\subseteq\mathbb{C}^M$,
$M\geq 1$, are connected and open sets. Then the following assertions are equivalent: }
\begin{enumerate}\itemsep2mm

\item \emph{for every subharmonic function $\varphi:\Omega_M\to\RE\cup\{-\infty\}$, the function
    $(\varphi\circ f):\Omega_N\to\RE\cup\{-\infty\}$ is subharmonic;}

\item \emph{for every harmonic function $h:\Omega_M\to\RE$, the function $(h\circ
    f):\Omega_N\to\RE$ is harmonic;}

\item \emph{$f$ has the following form:}

\smallskip

\begin{itemize}\itemsep2mm
\item[$(a)$] \emph{if $M\leq N$, then $f$ is constant or $f$ is , up to the composition with
a homotethetic map, the canonical orthogonal projection}
\[
\mathbb{C}^N= \mathbb{C}^M\times \mathbb{C}^{N-M}\to \mathbb{C}^M\, ,
\]
\emph{i.e., $f$ can be represented as}
\[
f(z)=cAz+w_0\, ,
\]
\emph{where $c\in\RE$,  $w_0\in\C^M$, and $A$ is a matrix as described in~(\ref{def_MxNunitary});}

\item[$(b)$] \emph{if $M>N$, then $f$ is constant.}

\end{itemize}

\end{enumerate}

\medskip

\noindent Theorem~B is quite different from Theorem~A. Note that we need to assume that
$f:\Omega_N\to\Omega_M$ is a holomorphic map. Otherwise the theorem, as stated, is not true (see
e.g. Example 1 on p.~114 in~\cite{Fuglede1978}, or the example on p.~297
in~\cite{gudmundssonsigurdsson}). The equivalence between (1) and (2) is due to Constantinescu and
Cornea (Corollary~3.2 in~\cite{ConstantinescuCornea}). Theorem~B in the case $M=1$ was observed by
Fuglede in~\cite{Fuglede1978}, as well as that condition (2) implies condition $(3b)$.  In 1980,
Baird and Eells generalized Fuglede's result for $M=1$ to the case when $\Omega_N$ is a K\"ahler
manifold and $\Omega_M$ is a Riemannian surface~\cite{BairdElles}. Alternative proofs of Theorem~B
can be found in
 Fu~\cite{Fu}, and Svensson~\cite{svensson1}. For further development in connection with Theorem~B
 see e.g.~\cite{Gudmundsson1, Gudmundsson2, Ou1, Ou2, svensson1, svensson2}. An influential article
 we would like to mention is~\cite{siu} from 1980 written by Siu. The correspondent result of
 Theorem~B in Euclidean spaces was proved by Fu~\cite{Fu} (see also~\cite{Sattayatham}).

\medskip

Next, we shall consider the generalization to the Caffarelli-Nirenberg-Spruck model. This model have
its origin in the article~\cite{CNS} by Caffarelli~\emph{et al.} from 1985. As indicated above,
harmonic morphisms gives a method of constructing minimal submanifolds. This should be compared to the authors of~\cite{CNS}
who provided a method of constructing special Lagrangian submanifolds, which are volume-minimizing
submanifolds introduced as an example within calibrated geometry~\cite{HarveyLawson}. In~\cite{B,Li}, the Caffarelli-Nirenberg-Spruck model
was adapted to a setting in complex space, and this attracted considerable attention. We would like to draw attention
to~\cite{DinewKolodziej, DinewLu, Lu, Cuong, Phong, WanWang}.

Let us now give a brief introduction to the Caffarelli-Nirenberg-Spruck model, and refer the reader
to Section~\ref{sec_CNS} for further details. Let $\Omega_N\subseteq\C^N$ be a connected, and open
set, and $1\leq k\leq N$. By $\mathbb C_{(1,1)}$ we denote the space of $(1,1)$-forms with constant
coefficients, and then we define
\[
\Gamma_k=\left\{\alpha\in \mathbb C_{(1,1)}: \alpha\wedge \beta^{N-1}\geq 0, \ldots ,
\alpha^k\wedge \beta ^{N-k}\geq 0   \right\},
\]
where $\beta=dd^c|z|^2$ is the canonical K\"{a}hler form on $\mathbb C^N$. Then, following
Caffarelli~\emph{et al.} we say that a subharmonic
function $u$ defined on $\Omega_N$ is called \emph{$k$-subharmonic} if
\[
dd^cu\wedge\alpha_1\wedge\dots\wedge\alpha_{k-1}\wedge\beta^{N-k}\geq 0
\]
holds in the sense of currents for all $\alpha_1,\ldots,\alpha_{k-1}\in \Gamma_k$. We denote the
set of all $k$-subharmonic functions defined on $\Omega_N$ by $\mathcal{SH}_k(\Omega_N)$.
Furthermore, we say that a function $u$ is \emph{$k$-pluriharmonic} on  $\Omega_N$, if $u$ and $-u$
are
$k$-subharmonic. The reason why we do not name these functions $k$-harmonic is to avoid confusion
with the notion of $p$-harmonic functions defined in nonlinear potential theory. Since $N$ is the
complex dimension of $\Omega_N$ we have that
\[
\mathcal{PSH}(\Omega_N)=\mathcal{SH}_N(\Omega_N)\subset\cdots\subset
\mathcal{SH}_1(\Omega_N)=\mathcal{SH}(\Omega_N)\, ,
\]
where $\mathcal{PSH}(\Omega_N)$ denotes the set of plurisubharmonic functions defined on
$\Omega_N$, and $\mathcal{SH}(\Omega_N)$ is the set of subharmonic functions defined on $\Omega_N$.
A typical example of a function in $\mathcal{SH}_k(\Omega_N)$, $1\leq k\leq N$, is
\[
G_k(z)=
\begin{cases}
-\|z\|^{2-\frac{2N}{k}} & \text{ if } k<N\\[2mm]
\log\|z\| & \text{ if } k=N.
\end{cases}
\]

\medskip

The third and final generalization is our classification of the holomorphic $(m,n)$-subharmonic
morphisms:

\medskip

\noindent \textbf{Theorem~C.} \emph{Set $M,N>1$, and let $1\leq m< M$, $1\leq n\leq N$, with
$m\leq n$. Also, let $\Omega_N\subseteq \mathbb C^N$ and $\Omega_M\subseteq \mathbb C^M$ be
connected and open sets, and assume that $f=(f_1,\ldots,f_M):\Omega_N\to \Omega_M$ is a holomorphic
mapping. Then the following two assertions are equivalent:}

\smallskip

\begin{enumerate}\itemsep2mm

\item  \emph{for every $m$-subharmonic function $\varphi:\Omega_M\to\RE\cup\{-\infty\}$, the
    function $(\varphi\circ f):\Omega_N\to\RE\cup\{-\infty\}$ is $n$-subharmonic, i.e. $f$ is a $(m,n)$-subharmonic morphism};

\item \emph{$f$ has the following form:}

\smallskip

\begin{itemize}\itemsep2mm
\item[$(a)$] \emph{if $M\leq N$ and $m=n$, then $f$ is constant or $f$ is, up to the
    composition with
a homotethetic map, the canonical orthogonal projection}
\[
\mathbb{C}^N= \mathbb{C}^M\times \mathbb{C}^{N-M}\to \mathbb{C}^M\, ,
\]
\emph{i.e., $f$ can be represented as}
\[
f(z)=cAz+w_0\, ,
\]
\emph{where $c\in\RE$,  $w_0\in\C^M$, and $A$ is a matrix as described in~(\ref{def_MxNunitary});}

\item[$(b)$] \emph{if $M\leq N$ and $m<n$, then $f$ is constant;}

\item[$(c)$] \emph{if $M>N$ and $m\leq n$, then $f$ is constant.}

\end{itemize}

\end{enumerate}

\medskip

If $m=n=1$, then condition (1) in Theorem~C is the same as condition (1) in Theorem~B.
Note that in Theorem~C we do \emph{not}
have any condition that states: for every  $m$-pluriharmonic function $h:\Omega_M\to\RE$, the
function $(h\circ f):\Omega_N\to\RE$
is $n$-pluriharmonic. This is in general impossible, since the set of $k$-subharmonic
functions, $2\leq k\leq N$,  defined on a connected, and open set $\Omega_N\subseteq\C^N$ is equal
to the set of all pluriharmonic functions defined on $\Omega_N$ (Proposition~\ref{prop_kPH}). Thus,
within the Caffarelli-Nirenberg-Spruck model there are not sufficiently many $k$-pluriharmonic functions, $2\leq k\leq N$.

 Finally, in Section~\ref{sec_open} we shall show that the case $m>n$ is more intricate, even when
 $m=2$, $n=1$. This is shown in Example~\ref{Finally_ex}, where we construct linear holomorphic
 mappings $f,g:\mathbb C^3\to \mathbb C^3$ with the property that $v\circ f\in \mathcal
 {SH}_1\left(\mathbb{C}^3\right)$ for every $v\in \mathcal {SH}_2\left(\mathbb{C}^3\right)$, but
 there is a function $u\in \mathcal{SH}_2\left(\mathbb{C}^3\right)$ such that $u\circ g\notin
 \mathcal {SH}_1\left(\mathbb{C}^3\right)$. We end this article with a necessary condition for a
 holomorphic mapping $f$ to
 satisfy condition (1) in Theorem~C, in the case when $m=2$, $n=1$ (Theorem~\ref{cor1}).

\medskip

 Our interest in Theorem~C did not begin with the interest of morphisms in  the
 Caffarelli-Nirenberg-Spruck model, but rather an attempt of finding a new way to construct
 $k$-subharmonic functions. We shall say a few words about this. Let $\Omega_N\subseteq\C^N$ be an
 open set, and let $\varphi:\Omega_{N}\to\RE\cup\{-\infty\}$ be a given upper semicontinuous
 function. Furthermore, $\sigma$ the normalized area measure defined on the unit circle
 $\mathbb{T}=\partial\mathbb{D}\subset\C$, and let $\mathcal{A}_{\Omega_N}$ be the set of all
 holomorphic maps $\bar{\mathbb{D}}\to\Omega_N$. Poletsky~\cite{Poletsky1,Poletsky2} and, Bu and
 Schachermayer~\cite{BuSchachermayer},  proved  independently the following equality:
\[
\sup \Big\{u(z): u\in \mathcal {PSH}(\Omega_N), u\leq \varphi\Big\}=\inf\left\{\int_{\mathbb T}
\varphi\circ f\,d\sigma: f\in \mathcal A_{\Omega_N}, f(0)=z \right\}\, .
\]
This equality is known as the \emph{Poisson disc formula}. The main reason why this is possible is
due to the rich structure of $\mathcal A_{\Omega_N}$ and Theorem~A. From Theorem~C it follows that
there is no longer any hope for a similar formula in the Caffarelli-Nirenberg-Spruck model.
Certainly, Theorem~B immediately implies that this is not possible, but only in  the case $\mathcal
{SH}_1$. For a similar phenomena in the Caffarelli-Nirenberg-Spruck model see e.g. Section~5
in~\cite{AhagCzyzHed1}.
On the other hand, there has been some progress obtaining a Poisson type formula for $\mathcal
{SH}_k$ using Choquet theory and representing measures, instead of Poletsky disks (see e.g.
Theorem~2.8 in~\cite{AhagCzyzHed2} and the references therein).

\medskip

For general information on potential theory see e.g.~\cite{armitage}, and for more information
about pluripotential theory we refer to~\cite{demailly_bok, K}.

\medskip

We would like to thank Marek Jarnicki, Maciej Klimek and Aron Persson for inspiring discussions, as
well as Sigmundur Gudmundsson and Ragnar Sigurdsson for their inspiring
article~\cite{gudmundssonsigurdsson}. We are very grateful to S\l awomir Dinew for letting us
incorporate his proof of Theorem~\ref{holo} (\cite{Dinew}).

\section{Proof of Theorem~A}\label{sec_firstpart}

Before we give a proof of Theorem~A we recall that for a $\mathcal C^2$-smooth (plurisubharmonic)
function $u$ defined on an open set $\Omega_N\subseteq \mathbb C^N$ its Levi form is defined by
\[
\mathcal L(u,a;X)=\sum_{j,k=1}^N\frac {\partial ^2u}{\partial z_j\partial \bar
z_k}\left(a\right)X_j\bar X_k\, ,
\]
where $a\in \Omega_N, X=(X_1,\ldots,X_N)\in \mathbb C^N$.

\medskip

\noindent \textbf{Theorem~A.} \emph{Let $f=(f_1,\ldots,f_M):\Omega_N\to\Omega_M$ be a function,
where $\Omega_N\subseteq\mathbb{C}^N$ and $\Omega_M\subseteq\mathbb{C}^M$, $N,M\geq 1$ are
connected and open sets with $N,M\geq 1$. Then the following assertions are equivalent:}
\begin{enumerate}\itemsep2mm

\item \emph{for every plurisubharmonic function $\varphi:\Omega_M\to\RE\cup\{-\infty\}$, the
    function $(\varphi\circ f):\Omega_N\to\RE\cup\{-\infty\}$ is plurisubharmonic;}

\item \emph{for every pluriharmonic function $h:\Omega_M\to\RE$, the function $(h\circ
    f):\Omega_N\to\RE$ is pluriharmonic;}

\item \emph{$f$ is holomorphic or $\bar f$ is holomorphic.}
\end{enumerate}
\begin{proof} The implication $(1)\Rightarrow(2)$ is immediate. Next we shall proceed with the
implication
$(3)\Rightarrow(1)$. Without loss of generality we can assume that the functions $\varphi$, and
$f$, are smooth, since the general case then follows by approximation. For a fixed $z\in \Omega_N$
we have for $1\leq i\leq N$ that
\[
\frac {\partial (\varphi\circ f)}{\partial z_j}(z) =\sum_{s=1}^M\left(\frac {\partial
\varphi}{\partial w_s}(f(z))\frac {\partial f_s}{\partial z_j}(z)+\frac {\partial \varphi}{\partial
\bar w_s}(f(z))\frac {\partial \bar f_s}{\partial z_j}(z) \right)\, ,
\]
and for $1\leq k\leq K$\, ,
\[
\frac {\partial (\varphi\circ f)}{\partial \bar z_k}(z) =\sum_{s=1}^M\left(\frac {\partial
\varphi}{\partial w_s}(f(z))\frac {\partial f_s}{\partial \bar z_k}(z)+\frac {\partial
\varphi}{\partial \bar w_s}(f(z))\frac {\partial \bar f_s}{\partial \bar z_k}(z) \right)\, ,
\]
which give us the following formula
\begin{multline}\label{pd}
\frac {\partial ^2(\varphi\circ f)}{\partial \bar z_k\partial z_j}(z)
=\sum_{r,s=1}^{M}\frac {\partial^2 \varphi}{\partial w_r\partial w_s}(f(z))\frac {\partial
f_r}{\partial \bar z_k}(z)\frac {\partial f_s}{\partial z_j}(z)\\ +
\sum_{r,s=1}^{M}\frac {\partial^2 \varphi}{\partial \bar w_r\partial \bar w_s}(f(z))\frac {\partial
\bar f_r}{\partial \bar z_k}(z)\frac {\partial \bar f_s}{\partial z_j}(z) +
\sum_{r,s=1}^{M}\frac {\partial^2 \varphi}{\partial \bar w_r\partial w_s}(f(z))\frac {\partial \bar
f_r}{\partial \bar z_k}(z)\frac {\partial f_s}{\partial z_j}(z)\\ +\sum_{r,s=1}^{M}\frac
{\partial^2 \varphi}{\partial w_r\partial \bar w_s}(f(z))\frac {\partial f_r}{\partial \bar
z_k}(z)\frac {\partial \bar f_s}{\partial z_j}(z)
+\sum_{r=1}^{M}\frac {\partial \varphi}{\partial \bar w_r}(f(z))\frac {\partial^2 \bar
f_r}{\partial \bar z_k\partial z_j}(z)\\ +\sum_{r=1}^{M}\frac {\partial \varphi}{\partial
w_r}(f(z))\frac {\partial^2 f_r}{\partial \bar z_k\partial z_j}(z)\, .
\end{multline}
For $X=(X_1,\ldots,X_N)\in \mathbb C^N$ let us define the following vectors
\[
\begin{aligned}
U=(U_1,\ldots,U_N), \ \text { where } \ U_r&=\sum_{j=1}^N\frac {\partial f_r}{\partial z_j}(z)X_j,
\\
V=(V_1,\ldots,V_N), \ \text { where } \ V_r&=\sum_{j=1}^N\frac {\partial \bar f_r}{\partial
z_j}(z)X_j\, .
\end{aligned}
\]
Then we have that
\begin{multline}\label{levi}
\mathcal L\left(\varphi\circ f,z;X\right)=\mathcal L\left(\varphi,f(z);U\right)+\mathcal
L\left(\varphi,f(z);V\right)+\\
+\sum_{r,s=1}^{M}\frac {\partial^2 \varphi}{\partial w_r\partial w_s}(f(z))U_r\bar V_s+
\sum_{r,s=1}^{M}\frac {\partial^2 \varphi}{\partial \bar w_r\partial \bar w_s}(f(z))\bar U_rV_s+\\
+\sum_{r=1}^{M}\frac {\partial \varphi}{\partial \bar w_r}(f(z))\mathcal L\left(\bar
f_r,z;X\right)+\sum_{r=1}^{M}\frac {\partial \varphi}{\partial  w_r}(f(z))\mathcal
L\left(f_r,z;X\right)\, .
\end{multline}
By the assumption that $f$ is holomorphic or $\bar f$ is holomorphic it follows that either $U=0$
or $V=0$. This implies that~(\ref{levi}) reduces to
\[
\mathcal L\left(\varphi\circ f,z;X\right)=\mathcal L\left(\varphi,f(z);U\right)+\mathcal
L\left(\varphi,f(z);V\right)\geq 0\, ,
\]
and we have obtained condition (1).

Let us now prove $(2)\Rightarrow(3)$. Assume that for every pluriharmonic function
$h:\Omega_M\to\RE$, the function $(h\circ f):\Omega_N\to\RE$ is pluriharmonic.
To gain sufficient information about $f$ we shall choose four specific pluriharmonic functions $h$.
Let $X\in \mathbb C^N$, and consider the following cases:
\begin{itemize}\itemsep2mm
\item[a)] if $h(w_1,\ldots,w_M)=2\operatorname{Re}(w_r)=w_r+\bar w_r$, then we get that
\begin{equation}\label{21}
\mathcal L\left(f_r,z;X\right)+\mathcal L\left(\bar f_r,z;X\right)=0
\end{equation}
for all $X\in \mathbb C^N$;

\item[b)] if $h(w_1,\ldots,w_M)=2\operatorname{Im}(w_r)=\frac 1i(w_r-\bar w_r)$,  then we get
    that
\begin{equation}\label{22}
\mathcal L\left(f_r,z;X\right)-\mathcal L\left(\bar f_r,z;X\right)=0
\end{equation}
for all $X\in \mathbb C^N$. Hence, by (\ref{21}) and (\ref{22}) we get
\[
\mathcal L\left(f_r,z;X\right)=\mathcal L\left(\bar f_r,z;X\right)=0,
\]
so $\operatorname {Re}f_r$ and $\operatorname {Im}f_r$ are pluriharmonic, and therefore real
analytic;
\item[c)] if $h(w_1,\ldots,w_M)=2\operatorname{Re}(w_rw_s)=w_sw_r+\bar w_r\bar w_s$, then formula
    (\ref{levi}) takes the following form
\begin{equation}\label{wec1}
0=\mathcal L\left(h\circ f,z;X\right)=U_r\bar V_s+\bar U_rV_s\, .
\end{equation}
\item[d)] if $h(w_1,\ldots,w_M)=2\operatorname{Im}(w_rw_s)=\frac 1i\left(w_sw_r-\bar w_r\bar
    w_s\right)$, then formula (\ref{levi}) takes the following form
\begin{equation}\label{wec2}
0=\mathcal L\left(h\circ f,z;X\right)=\frac 1i\left(U_r\bar V_s-\bar U_rV_s\right)\, .
\end{equation}
\end{itemize}
From (\ref{wec1}) and (\ref{wec2}) we now obtain that
\begin{equation}\label{wec3}
0=U_r\bar V_s=\left(\sum_{j=1}^N\frac {\partial f_r}{\partial
z_j}(z)X_j\right)\left(\sum_{j=1}^N\frac {\partial f_s}{\partial \bar z_j}(z)\bar X_j\right), \
\text { for all } \ r,s=1,\ldots, M.
\end{equation}
In (\ref{wec3}), if $X=e_j=(0,\ldots,1,\ldots,0)$ is taken as the $j$-th vector from the canonical
basis, then it holds that
\begin{equation}\label{a1}
\frac {\partial f_r}{\partial z_j}(z)\frac {\partial f_s}{\partial \bar z_j}(z)=0\, ,
\end{equation}
for all $z\in \Omega_N$, $r,s=1,\ldots,M$, and all $j=1,\ldots,N$.

Now, assume that $N\geq 2$. In (\ref{wec3}), if  $X=e_j+e_k$, $j\neq k$, then we get
\begin{multline}\label{a2}
0=\left(\frac {\partial f_r}{\partial z_j}(z)+\frac {\partial f_r}{\partial
z_k}(z)\right)\left(\frac {\partial f_s}{\partial \bar z_j}(z)+\frac {\partial f_s}{\partial \bar
z_k}(z)\right)
=\frac {\partial f_r}{\partial z_j}(z)\frac {\partial f_s}{\partial \bar z_k}(z)\\ +\frac {\partial
f_r}{\partial z_k}(z)\frac {\partial f_s}{\partial \bar z_j}(z)\, ,
\end{multline}
for all $z\in \Omega_N$, $r,s=1,\ldots,M$, and all $j,k=1,\ldots,N$. To proceed we take
$X=e_j+ie_k$, $j\neq k$, in (\ref{wec3}) and arrive at
\begin{multline}\label{a3}
0=\left(\frac {\partial f_r}{\partial z_j}(z)+i\frac {\partial f_r}{\partial
z_k}(z)\right)\left(\frac {\partial f_s}{\partial \bar z_j}(z)-i\frac {\partial f_s}{\partial \bar
z_k}(z)\right)=\\
=i\left(\frac {\partial f_r}{\partial z_j}(z)\frac {\partial f_s}{\partial \bar z_k}(z)-\frac
{\partial f_r}{\partial z_k}(z)\frac {\partial f_s}{\partial \bar z_j}(z)\right)\, ,
\end{multline}
for all $z\in \Omega_N$, $r,s=1,\ldots,M$, and all $j,k=1,\ldots,N$. From~(\ref{a2})
and~(\ref{a3}), it now follows that
\begin{equation}\label{a4}
\frac {\partial f_r}{\partial z_j}(z)\frac {\partial f_s}{\partial \bar z_k}(z)=0\, ,
\end{equation}
for all $z\in \Omega_N$, $r,s=1,\ldots,M$, and all $j,k=1,\ldots,N$. Since $f$ is real analytic,
(\ref{a4}) implies that
\begin{equation}\label{a5}
\frac {\partial f_r}{\partial z_j}(z)\equiv 0, \quad\text { or } \quad \frac {\partial
f_s}{\partial \bar z_k}(z)\equiv 0.
\end{equation}
If $f$ is a holomorphic mapping, then~(\ref{a4}) is satisfied. Now assume that $f$ is not
holomorphic mapping, then there exist $z\in \Omega_N$, $r\in \{1,\ldots,M\}$ and $j\in
\{1,\ldots,N\}$ such that
$\frac {\partial f_r}{\partial \bar z_j}(z)\neq 0$. Therefore, by (\ref{a4}) and (\ref{a5}), we get
that
\[
\frac {\partial f_s}{\partial z_k}(z)\equiv 0\, ,
\]
for all $s=1,\ldots,M$, and all $k=1,\ldots,N$. This means that $\bar{f}$ is holomorphic.

Finally if $N=1$, then using (\ref{a1}), and the argument above we get the desired conclusion.
\end{proof}

\section{The Caffarelli-Nirenberg-Spruck model and the proof of Theorem~C}\label{sec_CNS}

We start this section with introducing the necessary definitions and basic facts related to the
Caffarelli-Nirenberg-Spruck model in order to prove Theorem~C. For further information see
e.g.~\cite{SA,L}.

Let $\Omega_N\subseteq\C^N$ be a connected and open set, and $1\leq k\leq N$. By $\mathbb
C_{(1,1)}$ we denote the space of $(1,1)$-forms with constant coefficients, and then we define
\[
\Gamma_k=\left\{\alpha\in \mathbb C_{(1,1)}: \alpha\wedge \beta^{N-1}\geq 0, \ldots ,
\alpha^k\wedge \beta ^{N-k}\geq 0   \right\},
\]
where $\beta=dd^c|z|^2$ is the canonical K\"{a}hler form on $\mathbb C^N$.

\begin{definition}\label{m-sh} Assume that $\Omega_N \subseteq \C^N$ is a connected and open set,
and let $u$ be a subharmonic function defined on $\Omega_N$. Then we say that $u$ is
\emph{$k$-subharmonic} on $\Omega_N$, $1\leq k\leq N$, if the following inequality holds
\[
dd^cu\wedge\alpha_1\wedge\cdots\wedge\alpha_{k-1}\wedge\beta^{N-k}\geq 0\, ,
\]
in the sense of currents for all $\alpha_1,\ldots,\alpha_{k-1}\in \Gamma_k$. With
$\mathcal{SH}_k(\Omega_N)$ we denote the set of all $k$-subharmonic functions defined on
$\Omega_N$. We also say that $u$ is \emph{$k$-pluriharmonic} on  $\Omega_N$, if $u,-u\in \mathcal
{SH}_k(\Omega_N)$. We shall use the notation $\mathcal {PH}_k(\Omega_N)$ for the set of all
$k$-pluriharmonic functions defined on $\Omega_N$.
\end{definition}
\begin{remark}
Since $N$ is the dimension of $\Omega_N$ we have that
\[
\mathcal{PSH}(\Omega_N)=\mathcal{SH}_N(\Omega_N)\subset\cdots\subset\mathcal{SH}_1(\Omega_N)=\mathcal{SH}(\Omega_N)\,
.
\]
Furthermore, $\mathcal {PH}_1(\Omega_N)=\mathcal {H}(\Omega_N)$ and $\mathcal
{PH}_N(\Omega_N)=\mathcal {PH}(\Omega_N)$. Here,
$\mathcal {H}(\Omega_N)$ is the set of real-valued harmonic functions defined on $\Omega_N$, and
$\mathcal {PH}(\Omega_N)$ is the set of real-valued pluriharmonic functions defined on $\Omega_N$.
\end{remark}
Recall that the complex Hessian matrix of a $\mathcal C^2$-smooth function $u$ is given by
\[
\operatorname H_u(z)=\left [\frac {\partial ^2u}{\partial z_j\partial \bar
z_k}(z)\right]_{j,k=1}^N\, .
\]
Now we shall need the following notions. Let $\sigma_{k,N}$ be a $k$-elementary symmetric
polynomial of $N$-variable and $k\leq N$, i.e.,
\[
\sigma_{k,N}(x)=\sigma_{k,N}(x_1,\ldots,x_N)=\sum_{1\leq j_1<\cdots<j_k\leq
N}x_{j_1}\cdot\ldots\cdot x_{j_k}\, .
\]
Let us define the following sets
\[
\Lambda_{k,N}=\left \{x\in \mathbb R^N: \sigma_{l,N}(x)\geq 0,\, l=1,\ldots,k \right\}.
\]
It can be proved that a $\mathcal C^2$ function $u$ is $k$-subharmonic if, and only if,
\[
\lambda(z)=(\lambda_1(z),\ldots,\lambda_N(z))\in \Lambda_{k,N}\, ,
\]
where $\lambda_1(z),\ldots,\lambda_N(z)$ are eigenvalues of the Hessian matrix $\operatorname
H_u(z)$. This is one of
the key ingredients in the proof of Theorem~C.

\begin{remark}
Let $\lambda_1,\dots,\lambda_N$ be eigenvalues of a square matrix $A=[a_{ij}]$ of dimension $N$.
Let us recall that
\[
\operatorname{det}\left (I+tA\right)=t^N+\sum_{k=1}^N\sigma_{k,N}(A)t^{N-k},
\]
where $I$ is the identity matrix and $\sigma_{k,N}(A)=\sigma_{k,N}(\lambda_1,\ldots,\lambda_N)$.
Furthermore, by the
Faddeev-Le Verrier algorithm, we have the following:
\[
\sigma_{k,N}(A)=\frac {1}{k!}\operatorname{det}\left[
  \begin{array}{ccccccc}
   T_1 & k-1 & 0 & \dots & 0 & 0 & 0 \\
   T_2 & T_1 & k-2 & \dots & 0 & 0 & 0 \\
   T_3 & T_2 & T_1 & \dots & 0 & 0 & 0 \\
   \vdots & \vdots & \vdots & \ddots & \vdots & \vdots & \vdots \\
   T_{k-2} & T_{k-3} & T_{k-4} & \dots & T_1 & 2 & 0 \\
   T_{k-1} & T_{k-2} & T_{k-3} & \dots & T_2 & T_1 & 1 \\
   T_{k} & T_{k-1} & T_{k-2} & \dots & T_3 & T_2 & T_1 \\
  \end{array}
\right]  ,
\]
where $T_k=\operatorname{tr}(A^k)$. In particular, for $k=2$ we have that
\begin{equation}\label{sigma2}
\sigma_{2,N}(A)=\frac 12 \left(\left(\operatorname{tr}(A)\right)^2 -
\operatorname{tr}\left(A^2\right)\right)=\frac12\sum_{j\neq
k}\left(a_{jj}a_{kk}-a_{jk}a_{kj}\right).
\end{equation}
We shall make use of~(\ref{sigma2}) in the proof of Theorem~\ref{holo}. For further information
about the Faddeev-Le Verrier algorithm see e.g.~\cite{hou}, and the references therein.
\end{remark}

Next, in Proposition~\ref{prop_kPH}, we shall characterize the $k$-pluriharmonic functions.

\begin{proposition}\label{prop_kPH} Assume that $\Omega_N \subseteq \C^N$ is a connected and open
set. Then,
$\mathcal {PH}_k(\Omega_N)=\mathcal {PH}(\Omega_N)$ for all $2\leq k\leq N$.
\end{proposition}

\begin{proof} We prove the case $k=2$. The other cases can be proved in a similar way. First note that if $u\in \mathcal {PH}_2(\Omega_N)$, then $u$ is a harmonic function,
and therefore it is smooth. If
$\lambda_1(z),\ldots,\lambda_N(z)$ are the eigenvalues of the Hessian matrix $\operatorname
H_u(z)$, then
$-\lambda_1(z),\ldots,-\lambda_N(z)$ are the eigenvalues of the Hessian matrix $\operatorname
H_{-u}(z)$. Therefore,  we get that
\[
\lambda_1(z)+\ldots+\lambda_N(z)=0, \quad \text { and } \ \sum_{j<k}\lambda_j(z)\lambda_k(z)\geq
0\, .
\]
Hence,
\[
0=(\lambda_1(z)+\ldots+\lambda_N(z))^2=\sum_{j}\lambda^2_j(z) +
2\sum_{j<k}\lambda_j(z)\lambda_k(z)\geq 0\, ,
\]
which means that $\lambda_1(z)=\ldots=\lambda_N(z)=0$. Thus, $u$ is a pluriharmonic function.
\end{proof}

In Proposition~\ref{cone} some elementary properties of $\Lambda_{k,N}$ is presented.
Proposition~\ref{cone} shall be used in Lemma~\ref{lem4} as well as in the proof of Theorem~C.

\begin{proposition}\label{cone}
The sets $\Lambda_{k,N}$ have the following properties:
\begin{enumerate}\itemsep2mm
\item $\mathbb R_+^N=\Lambda_{N,N}\subset \Lambda_{N-1,N}\subset\cdots \subset \Lambda_{1,N}$;

\item $\Lambda_{k,N}$ is a closed convex cone for all $1\leq k\leq N$;

\item if $x=(x_1,\ldots,x_N)\in \Lambda_{k,N}$, then
\[
(x_1,\ldots,x_N,\overbrace{0,\ldots,0}^{l})\in \Lambda_{k,N+l}\, ;
\]
\item $(1,\ldots,1,x)\in \Lambda_{k,N}$ if, and only if, $x\geq 1-\frac Nk$. In particular,
\[
\left(1,\ldots,1,1-\frac Nk\right)\in \Lambda_{k,N}\, .
\]
\end{enumerate}
\end{proposition}
\begin{proof}
For properties (1), (2) and (3)  see e.g.~\cite{B}. Property~(4) follows from that
\[
\sigma_{l,N}\left(1,\ldots,1,x\right)=\binom{N}{l}\frac {N+l(x-1)}{N}\geq 0\, ,
\]
for all $l=1,\ldots,k$ if, and only if, $x\geq 1-\frac Nk$.
\end{proof}

In Lemma~\ref{lem4} as well as in the proof of Theorem~C we shall make use of the so called
\emph{Hadamard product} for vectors in $\mathbb{R}^L$. Recall that for
$x=(x_1,\ldots,x_L),y=(y_1,\ldots,y_L)\in \mathbb R^L$ we define
\[
x\odot y=(x_1y_1,\ldots,x_Ly_L)\, .
\]

The following lemma is of significant importance in the proof of Theorem~C.

\begin{lemma}\label{lem4}
Let $y=(y_1,\ldots,y_L)\in \mathbb R_+^{L}$ and let $k\leq l$, $k<K$. Then the following statements are
true:
\begin{enumerate}\itemsep2mm
\item if $L<K$ and for any $x\in \Lambda_{k,K}$ holds
\[
y\odot x\big|_L=(y_1 x_1,\ldots,y_L x_L)\in \Lambda_{l,L}\, ,
\]
where $x\big|_L=(x_1,\ldots,x_L)$, then it holds that $y_1=\ldots=y_L=0$;

\item if $L\geq K$ and  for any $x\in \Lambda_{k,K}$ holds
\[
y\odot x\big|^L=(y_1x_1,\ldots,y_K x_K,0,\ldots,0)\in \Lambda_{l,L}\, ,
\]
where $x\big|^L=(x_1,\ldots,x_K,0,\ldots,0)\in \mathbb R^L$, then it holds that
\[
\begin{cases}
y_1=\ldots=y_L & \text{ if } k=l\, , \\

y_1=\ldots=y_L=0 & \text{ if } k<l\, .
\end{cases}
\]
\end{enumerate}
\end{lemma}
\begin{proof} \emph{Property (1):} Assume that $L<K$. Let $y=(y_1,\ldots,y_L)\in \mathbb R^L_+$.
Without loss of generality we can always assume that $0\leq y_1\leq y_2\leq \cdots \leq y_L$.

\medskip

\emph{The case $k=l$.} First assume that $y_j\neq 0$, for $j=1,\ldots, L$. Take
\[
x=\left(\overbrace{1,\ldots,1}^{k},\overbrace{0,\ldots,0}^{K-k-1},1-\frac {k+1}{k},\right)\in
\Lambda_{k,K}\, ,
\]
and note that
\begin{multline*}
\sigma_{k,L}(y\odot x|_L)=\sigma_{k,L}\left(y_1,\ldots,y_k,0,\ldots,0,\left(1-\frac
{k+1}{k}\right)y_L\right) \\
=y_1\cdots y_{k}y_L\left(\frac {1}{y_L}+\left(1-\frac {k+1}{k}\right)\sum_{j=1}^k\frac
{1}{y_j}\right)\\
\leq y_1\cdots y_{k}y_L\left(\frac {1}{y_L}-\frac {1}{k}\sum_{j=1}^{k}\frac {1}{y_L}\right)=0\, .
\end{multline*}
Next, observe that $y\odot x\big|_L\in \Lambda_{k,L}$ if, and only if, $y_1=\ldots=y_{k}=y_L$.
Since one can permutate coordinates in the vector $x$, then any $k+1$ coordinates of the vector $y$
must be equal. Therefore, it is  obtained that $y_1=\ldots=y_L=a$. To continue, let
\[
x=\left(1-\frac {K}{k},\overbrace{1,\dots,1}^{K-1}\right)\in \Lambda_{k,K}\, ,
\]
and therefore we have that
\[
y\odot x\big|_L=\left(\left(1-\frac {K}{k}\right)a,\overbrace{a,\dots,a}^{L-1}\right)\, .
\]
Hence,
\begin{multline*}
\sigma_k\left(\left(1-\frac {K}{k}\right)a,\overbrace{a,\dots,a}^{L-1}\right)=a^k\left(\binom
{L-1}{k}\cdot 1+\binom{L-1}{k-1}\cdot 1\cdot\left(1-\frac Kk\right)\right)\\
=a^k\binom{L}{k}\frac {L-K}{L}\leq 0\, .
\end{multline*}
This means that $y\odot x\big|_L\in \Lambda_{k,L}$ if, and only if, $a=0$. Furthermore, assume that
the first $s$ coordinates of the vector $y$ vanish, i.e., $y=(0,\dots,0,y_{s+1},\dots,y_L)$. Assume
additionally that $s\geq k$, and for $j>s$ let
\[
x=\left(\overbrace{1,\dots,1}^{s},0,\dots,0,\overbrace{1-\frac
{s+1}{k}}^{j},0,\dots,0\right)\in\Lambda_{k,L}\, .
\]
Then, we have that
\[
y\odot x|_L=\left(0,\dots,0,\overbrace{\left(1-\frac {s+1}{k}\right)y_j}^{j\text{-th
position}},0,\dots,0\right)\, ,
\]
so $y\odot x|_L\in \Lambda_{k,L}$ if, and only if, $y_j=0$. Next, assume that $s<k$ and take
\[
x=\left(\overbrace{1,\dots,1}^{k},1-\frac {k+1}{k},0,\dots,0\right)\in\Lambda_{k,L}\, .
\]
From this it follows that
\[
y\odot x|_L=\left(\overbrace{0,\dots,0}^{s},y_{s+1},\dots,y_k,\left(1-\frac
{k+1}{k}\right)y_{k+1},0,\dots,0\right)\, ,
\]
and then
\[
\sigma_{k-s}(y\odot x|_L)=y_{s+1}\cdots y_{k+1}\left(1-\frac {k+1}{k}\right)\leq 0\, .
\]
Therefore one of  $y_{s+1},\ldots,y_{k+1}$ must be zero, say $y_p$. We are adding this zero, $y_p$,
to the first $s$ zero coordinates, and then we repeat
the argument above. From this procedure we then obtain inductively that $y_1=\ldots=y_L=0$.

\medskip

\emph{The case $k<l$.} From Proposition~\ref{cone} (1) we have that $\Lambda_{l,L}\subset
\Lambda_{k,L}$, and therefore by the previous case we have that
$y_1=\ldots=y_L=0$.

\bigskip

\noindent\emph{Property (2):} Assume that $L\geq K$.

\medskip

\emph{The case $k=l$.} We can proceed as the case $k=l$ in Property (1) to get that
$y_1=\ldots=y_n=a$.

\medskip

\emph{The case $k<l$.} From Proposition~\ref{cone} (1) it follows that $\Lambda_{l,L}\subset
\Lambda_{k,L}$, and therefore by the  previous case we have that $y_1=\ldots=y_L=a$. If we then
take
\[
x=\left (1,\dots,1,1-\frac Kk\right)\in \Lambda_{k,K}\, ,
\]
we get that
\[
y\odot x\big|^L=\left(a,\dots,a,a\left(1-\frac Kk\right),0,\dots,0\right)\in \Lambda_{l,L}
\]
if, and only if, $a=0$.
\end{proof}

Example~\ref{ex3} shows that we can not have a corresponding result to Lemma~\ref{lem4} in the case
$k>l$.

\begin{example}\label{ex3} In Lemma~\ref{lem4}, consider the case when $K=3$, $L=4$, $k=2$ and
$l=1$. Define,
\[
\Omega=\left\{(x_1,x_2,x_3)\in\RE^3: x_1+x_2+x_3\geq 0,\, x_1x_2+x_2x_3+x_3x_1\geq 0 \right\}\, .
\]
Set $y=(1,4,9,1)$, and consider the function $f$ defined on $\Omega$ by
\[
f(x_1,x_2,x_3)=x_1+4x_2+9x_3\, .
\]
Then $f\geq 0$ on $\Omega$. On the other hand, for $y=(1,1,9,1)$ the function $g$ defined on
$\Omega$ by
\[
g(x_1,x_2,x_3)=x_1+x_2+9x_3
\]
attains both positive and negative values. This shows that there is no similar characterization as
in Lemma~\ref{lem4} for the case $k>l$.\hfill{$\Box$}
\end{example}

Next we introduce the formal definition of $(m,n)$-subharmonic morphisms.

\begin{definition}\label{def_sunmorph}
Let $1\leq m\leq M$, $1\leq n\leq N$, and let $\Omega_N\subseteq \mathbb C^N$ and $\Omega_M\subseteq \mathbb C^M$ be
connected and open sets. We say that a function $f=(f_1,\ldots,f_M):\Omega_N\to \Omega_M$ is a \emph{$(m,n)$-subharmonic morphism} if for every $m$-subharmonic function $\varphi:\Omega_M\to\RE\cup\{-\infty\}$, the
function $(\varphi\circ f):\Omega_N\to\RE\cup\{-\infty\}$ is $n$-subharmonic.
\end{definition}

In connection with Definition~\ref{def_sunmorph} we give in Proposition~\ref{prop} some elementary properties of $(m,n)$-subharmonic morphisms.

\begin{proposition}\label{prop}
Let $1\leq m\leq M$, $1\leq n\leq N$ and let $\Omega_N\subseteq \mathbb C^N$ and $\Omega_M\subseteq \mathbb C^M$ be
connected and open sets. The functions below are defined on $\Omega_N$ with range in $\Omega_M$. Then we have that
\begin{enumerate}\itemsep2mm
\item $(1,1)$-subharmonic morphisms are precisely the harmonic morphisms;

\item $(M,N)$-subharmonic morphisms are precisely the family of holomorphic and anti-holomorphic mappings;

\item if $m'\geq m$ and $n'\leq n$, then every $(m,n)$-subharmonic morphism is also a $(m',n')$-subharmonic morphism;

\item every holomorphic or anti-holomorphic mapping is also a $(M,n)$-subharmonic morphism.
\end{enumerate}
\end{proposition}
\begin{proof}
Assertion $(1)$ is a consequence of Theorem~B, while assertion $(2)$ is a consequence of Theorem~A. Property $(3)$ follows from the inclusions
\[
\mathcal {SH}_{m'}(\Omega_M)\subset \mathcal {SH}_m(\Omega_M)  \quad \text{and} \quad \mathcal {SH}_{n'}(\Omega_N)\supset \mathcal {SH}_n(\Omega_N)\, .
\]
Finally, property $(4)$ follows from $(2)$ and $(3)$.
\end{proof}

Next we shall give a proof of Theorem~C. To simplify the presentation we shall use the following
notation for the complex gradient of a function
$g:\mathbb C^N\to \mathbb C$:
\[
\nabla_{\mathbb C}g=\left(\frac {\partial g}{\partial z_1},\ldots,\frac {\partial g}{\partial
z_N}\right)\, .
\]
If $M\leq N$, then recall that a function $f:\mathbb{C}^N\to \mathbb{C}^M$  that is, up to the composition with
a homotethetic map, the canonical orthogonal projection
\[
\mathbb{C}^N= \mathbb{C}^M\times \mathbb{C}^{N-M}\to \mathbb{C}^M
\]
can be represented as
\[
f(z)=cAz+w_0\, ,
\]
where $c\in\RE$,  $w_0\in\C^M$, and $A$ is a $M\times N$  matrix such that there exist
$X_1,\ldots,X_M\in \mathbb C^N$ with $\|X_1\|=\ldots=\|X_M\|$, $\langle X_j,X_k\rangle=0$ for all
$j\neq k$, and
\begin{equation}\label{def_MxNunitary}
A=\left[
\begin{array}{c}
X_1 \\
\vdots \\
X_M \\
\end{array}
\right]\, .
\end{equation}

We shall now prove Theorem~C.

\medskip

\noindent \textbf{Theorem~C.} \emph{Set $M,N>1$, and let $1\leq m<M$, $1\leq n\leq N$, with
$m\leq n$. Also, let $\Omega_N\subseteq \mathbb C^N$ and $\Omega_M\subseteq \mathbb C^M$ be
connected and open sets, and assume that $f=(f_1,\ldots,f_M):\Omega_N\to \Omega_M$ is a holomorphic
mapping. Then the following two assertions are equivalent:}

\smallskip

\begin{enumerate}\itemsep2mm

\item  \emph{$f$ is a $(m,n)$-subharmonic morphism;}

\item \emph{$f$ has the following form:}

\smallskip

\begin{itemize}\itemsep2mm
\item[$(a)$] \emph{if $M\leq N$ and $m=n$, then $f$ is constant or $f$ is, up to the
    composition with
a homotethetic map, the canonical orthogonal projection}
\[
\mathbb{C}^N= \mathbb{C}^M\times \mathbb{C}^{N-M}\to \mathbb{C}^M\, .
\]

\item[$(b)$] \emph{if $M\leq N$ and $m<n$, then $f$ is constant;}

\item[$(c)$] \emph{if $M>N$ and $m\leq n$, then $f$ is constant.}

\end{itemize}

\end{enumerate}
\begin{proof}  Before we start note that if $f$ is holomorphic and $\varphi$ is a smooth function,
as in the statement of the theorem, then we have that
\[
\frac {\partial^2 (\varphi\circ f)}{\partial z_j\partial \bar z_k}=\sum_{r,s=1}^M\frac {\partial^2
\varphi}{\partial w_r\partial \bar w_s}\frac {\partial f_r}{\partial z_j}\overline{\frac {\partial
f_s}{\partial z_k}}\, .
\]
If we for this proof adopt the notation that $\times$ is matrix multiplication, and  $A^{*}$
denotes the conjugate transpose of the matrix $A$, then we get that
\begin{equation}\label{mat}
\operatorname H_{\varphi\circ f}=[\nabla_{\mathbb C}f_1\cdots \nabla_{\mathbb C}f_M ]\times
\operatorname H_{\varphi}\times [\nabla_{\mathbb C}f_1\cdots \nabla_{\mathbb C}f_M ]^{*}\, .
\end{equation}
Note that we also shall use $\times$ to denote the dimension of matrices. Throughout this proof we
can assume that $\varphi$ is always smooth, since the general conclusions we want to obtain always
follows by approximation.

\medskip

$(2)\Rightarrow(1)$: If $f$ is a constant function, then the implication is immediately true.
Therefore, we can assume that $m=n$, and that the function $f=(f_1,\ldots,f_M):\Omega_N\to
\Omega_M$ can be written as
\[
f(z)=cAz+w_0\, ,
\]
where $c\in\RE$,  $w_0\in\C^M$, and $A$ is a matrix as described in~(\ref{def_MxNunitary}). To deduce $(1)$ in the case $M\leq N$, and $m=n$, is then straightforward.

\medskip

\noindent $(1)\Rightarrow(2)$:  Assume that for every $m$-subharmonic function
$\varphi:\Omega_M\to\RE\cup\{-\infty\}$, the function $(\varphi\circ
f):\Omega_N\to\RE\cup\{-\infty\}$ is $n$-subharmonic.

\medskip

\emph{The case $M=N$.} By the polar decomposition theorem (see e.g.~\cite{Cooperstein}) there exist
a unitary matrix $U$, and a positive semi-definite Hermitian matrix $B$ such that
\[
[\nabla_{\mathbb C}f_1\cdots \nabla_{\mathbb C}f_M]=B\times U\, .
\]
We can next diagonalize $B$, i.e., there exist matrices $C$, and $D$ such that $B=C\times D\times
C^{-1}$, where $C$ is unitary, $D$
is a diagonal matrix given by
\[
D=\left[
  \begin{array}{ccc}
   \mu_1 & \dots & 0 \\
   \vdots & \ddots & \vdots \\
   0 & \dots & \mu_M \\
  \end{array}
\right]\, ,
\]
and $0\leq \mu_M\leq \mu_{M-1}\leq \cdots\leq \mu_1$ are the (real) eigenvalues of the matrix $B$.
Then by (\ref{mat}) we can deduce
that
\begin{equation}\label{mat2}
\operatorname H_{\varphi\circ f}=C\times D\times C^{-1}\times U\times
\operatorname{H_{\varphi}}\times U^*\times C \times D\times  C^{-1}\, .
\end{equation}
If we then choose a smooth $m$-subharmonic function $\varphi$ such that
\[
\operatorname H_{\varphi}=U^*\times C\times\tilde{D}\times C^{-1}\times U,
\]
where $\tilde{D}$ is the diagonal matrix with the (real) eigenvalues $\lambda_1\leq\cdots\leq
\lambda_N$  of $\operatorname H_{\varphi}$ on the diagonal, then expression (\ref{mat2}) will have
the form
\[
\operatorname H_{\varphi\circ f}=C\times\left[
  \begin{array}{ccc}
   \mu_1^2\lambda_1 & \dots & 0 \\
   \vdots & \ddots & \vdots \\
   0 & \dots & \mu_N^2\lambda_N \\
  \end{array}
\right]\times C^{-1}\, .
\]
Hence, $\mu_1^2\lambda_1,\dots,\mu_N^2\lambda_N$ are eigenvalues of the Hessian matrix
$\operatorname H_{\varphi\circ f}$.
Thus, for every $\lambda=(\lambda_1,\dots,\lambda_N)\in \Lambda_{m,N}$ we get that
\[
\mu^2\odot\lambda=(\mu_1^2\lambda_1,\dots,\mu_N^2\lambda_N)\in \Lambda_{n,N}\, .
\]
By Lemma~\ref{lem4} we then can conclude that:
\begin{itemize}\itemsep2mm
\item[$(i)$] if $m=n$, then $\mu_1=\ldots=\mu_N=a$. Therefore, the proof
    of this case is then finished by the Lemma in~\cite{gudmundssonsigurdsson};

\item[$(ii)$] if $m<n$, then $\mu_1=\ldots=\mu_N=0$. Thus, $f$ must be a constant function.
\end{itemize}

\medskip

\emph{The case $M\neq N$.} Let $r$ denote the rank of the $N\times M$ matrix
\[
[\nabla_{\mathbb C}f_1\cdots \nabla_{\mathbb C}f_M]\, ,
\]
and $0\leq s_r\leq s_{r-1}\leq \cdots\leq s_1$ be the singular values. Remember that singular
values are in general not the same as eigenvalues. Next, let $S$ be the $N\times M$ matrix whose
$(i,j)$-entry is $s_i$ if $i=j\leq r$ and $0$ otherwise. Then,  by the singular value decomposition
theorem (see e.g.~\cite{Cooperstein}), there exists an $N \times N$ unitary matrix $V$ and an $M
\times M$ unitary matrix $W$ such that
\begin{equation}\label{e1}
[\nabla_{\mathbb C}f_1\cdots \nabla_{\mathbb C}f_M]=V\times S\times W^*.
\end{equation}
Next, make an orthogonal change of variables in $\mathbb C^N$. We do this through the following
function
\[
\tilde f(z)=f\left(\left(V^T\right)^{-1}z\right),
\]
where $V^T$ denote the transpose of $V$. Then we get that
\begin{equation}\label{e5}
\left[\nabla_{\mathbb C}\tilde f_1\cdots \nabla_{\mathbb C}\tilde f_M\right]=S\times W^* ,
\end{equation}
and by the case $M\neq N$ in this implication, we have that for every $\varphi\in \mathcal
{SH}_m(\Omega_M)$ such that $\varphi\circ f$ is $n$-subharmonic, then we also know that
$\varphi\circ \tilde f$ is $n$-subharmonic. By employing~(\ref{mat}) we arrive at
\begin{equation}\label{e3}
\operatorname H_{\varphi\circ \tilde f}=S\times W^*\times \operatorname H_{\varphi} \times W\times
S^T.
\end{equation}
Now take a smooth function $\varphi\in \mathcal {SH}_{m}(\Omega_M)$ with the property that
\begin{equation}\label{e4}
\operatorname H_{\varphi}=W\times \tilde{D}\times W^* ,
\end{equation}
where $\tilde{D}$ is the diagonal matrix with the eigenvalues $\lambda_1,\dots,\lambda_{M}$ of the
Hessian matrix $\operatorname H_{\varphi}$ on its diagonal. Using~(\ref{e3}) and~(\ref{e4}) we
conclude that for every $(\lambda_1,\dots,\lambda_M)\in \Lambda_{m,M}$ we have that
\[
\begin{cases}
(s_1^2\lambda_1,\dots,s_N^2\lambda_N)\in \Lambda_{n,N} & \text { when } N<M; \\[2mm]
(s_1^2\lambda_1,\dots,s_M^2\lambda_M,0,\dots,0)\in \Lambda_{n,N}& \text { when } N>M\, .
\end{cases}
\]
Once again relying on Lemma~\ref{lem4} we get that:
\begin{itemize}\itemsep2mm
\item[$(i)$] if $M<N$, and $m=n$, then $s_1=\ldots=s_N=a$. The proof is then
    finished by the Lemma in~\cite{gudmundssonsigurdsson};

\item[$(ii)$] in the other cases we get that $s_1=\ldots=s_N=0$, so $f$ must be a constant
    function.
\end{itemize}
\end{proof}

We shall end Section~\ref{sec_CNS} with the following theorem that is related to condition~(1) in
Theorem~C. Observe that the function $f$ in Theorem~C must be assumed to be holomorphic (see
e.g. Example~1 on p.~114 in~\cite{Fuglede1978}, or the example on p.~297
in~\cite{gudmundssonsigurdsson}), while the function $F$ in Theorem~\ref{holo} need no such
assumption. Theorem~\ref{holo} is not true for $n=1$, since for example the M\"{o}bius
transform preserves the class of subharmonic functions. The proof of Theorem~\ref{holo} is due to
Dinew~\cite{Dinew}.

\begin{theorem}\label{holo}
Set $1\leq m\leq M$, $1<n\leq N$. Also, let $\Omega_N\subseteq \mathbb C^N$, and $\Omega_M\subseteq
\mathbb C^M$, be connected and open sets, and assume that $F=(F^1,\ldots,F^M):\Omega_N\to \Omega_M$ is $(m,n)$-subharmonic morphism. Then, $F$ is holomorphic or
$\bar{F}$ is holomorphic.
\end{theorem}
\begin{proof} The function $F$ must be real analytic. To prove this take
$\varphi(w_1,\ldots,w_M)=\pm \operatorname {Re}w_k, \pm \operatorname{Im}w_k$, $k=1,\ldots,M$, then
we get functions $\pm \operatorname {Re}F^k, \pm \operatorname{Im}F^k$ that are $n$-subharmonic, so
in particular subharmonic. This means that $\operatorname {Re}F^k, \operatorname{Im}F^k$ are
harmonic and therefore real analytic.

Without loss of generality it is sufficient to prove this theorem for $n=2$. For simplicity let us
for any smooth function $h$ use the following notation
\[
h_j=\frac {\partial h}{\partial z_j}(z)\quad \text{and} \quad h_{\bar j}=\frac {\partial
h}{\partial {\bar z}_j}(z)\, .
\]
Without loss of generality we can assume that $\varphi$ is a smooth $m$-subharmonic function
defined on $\Omega_M$, since
the general case follows by approximation. Also, let $z\in \Omega_M$. Then by (\ref{pd}) we have
that
\begin{multline}\label{sigma1}
\sigma_{1,N}(\varphi\circ F)(z)=\sum_{j=1}^N(\varphi\circ
F)_{j\overline{j}}(z)=\sum_{j=1}^N\sum_{k=1}^M\left(\varphi_kF_{j\overline{j}}^k+\varphi_{\overline{k}}\overline{F}_{j\overline{j}}^k\right)
\\
+\sum_{j=1}^N\sum_{k,s=1}^M\left (\varphi_{ks}F_{\overline{j}}^kF_{j}^s+\varphi_{k\overline{s}}
F_{\overline{j}}^k\overline{F}_{j}^s+
\varphi_{\overline{k}s}\overline{F}_{\overline{j}}^kF_{j}^s+\varphi_{\overline{ks}}
\overline{F}_{\overline{j}}^k\overline{F}_{j}^s\right )\geq 0\, .
\end{multline}
If we insert $\varphi_1(w_1,\dots,w_M)=\operatorname{Re}\left(a_kw_k\right)$ into (\ref{sigma1}),
for arbitrary $a_k\in \mathbb C$, and $k=1,\dots,M$, then we get that
\[
\sigma_{1,N}(\varphi_1\circ F)=\frac
12\sum_{j=1}^Na_kF_{j\overline{j}}^k+\overline{a_k}\overline{F}_{j\overline{j}}^k=\operatorname{Re}\left(\sum_{j=1}^Na_kF_{j\overline{j}}^k\right)\geq
0\, .
\]
Since $a_k$ was chosen arbitrarily it follows that
\begin{equation}\label{eq2}
\sum_{j=1}^NF_{j\overline{j}}^k=0,\quad k=1,\ldots,M\, .
\end{equation}
Next, let $\varphi_2(w_1,\dots,w_M)=\operatorname{Re}\left(a_{ks}w_kw_s\right)$ for arbitrary
$a_{ks}\in \mathbb C$, and $k,s=1,\dots,M$. Then using (\ref{eq2}) we can simplify~(\ref{sigma1})
to the following form
\begin{align*}
&\sigma_{1,N}(\varphi_2\circ F)=\frac
12\sum_{j=1}^N\left(a_{ks}F_{\overline{j}}^kF_j^s+\overline{a}_{ks}\overline{F}_{\overline{j}}^k\overline{F}_j^s+
a_{ks}F_{\overline{j}}^sF_j^k+\overline{a}_{ks}\overline{{F}_j}^k\overline{F}_{\overline{j}}^s\right)=\\
&=\operatorname{Re}\left(a_{ks}\sum_{j=1}^NF_{\overline{j}}^kF_j^s+F_{\overline{j}}^sF_j^k\right)\geq
0\, ,
\end{align*}
and since $a_{k,s}$ were arbitrary we get that
\begin{equation}\label{eq3}
\sum_{j=1}^N\left(F_{\overline{j}}^kF_j^s+F_{\overline{j}}^sF_j^k\right)=0,\quad  k,s=1,\cdots,M\,
.
\end{equation}
Finally,~(\ref{sigma1}) simplifies to
\begin{equation}\label{eq4}
\sigma_{1,N}(\varphi\circ
F)=\sum_{j=1}^N\sum_{k,s=1}^M\left(\varphi_{k\overline{s}}F_{\overline{j}}^k\overline{F}_j^s+
\varphi_{\overline{k}s}\overline{F}_{\overline{j}}^kF_j^s\right)\geq 0\, .
\end{equation}
Next, we shall consider $\sigma_{2,N}(\varphi\circ F)$. By~(\ref{sigma2}) we get that
\begin{equation}\label{eq5}
\sigma_{2,N}(\varphi\circ F)=\frac{1}{2}\sum_{i,j=1}^{N}\left((\varphi\circ
F)_{i\overline{i}}(\varphi\circ
F)_{j\overline{j}}-(\varphi\circ F)_{i\overline{j}}(\varphi\circ
F)_{j\overline{i}}\right)\geq 0\, ,
\end{equation}
and then from (\ref{eq4}) it follows that the first term in~(\ref{eq5}) is equal to
\begin{multline}\label{eq6}
\sum_{i,j=1}^N\sum_{k,s=1}^M\sum_{p,q=1}^M \bigg(\varphi_{k\overline{s}}\varphi_{p\overline{q}}
F_{\overline{j}}^k\overline{F}_j^sF_{\overline{i}}^p\overline{F}_i^q+
\varphi_{k\overline{s}}\varphi_{\overline{p}q}
F_{\overline{j}}^k\overline{F}_j^s\overline{F}_{\overline{i}}^pF_i^q \\
+\varphi_{\overline{k}s}\varphi_{p\overline{q}}
\overline{F}_{\overline{j}}^kF_j^sF_{\overline{i}}^p\overline{F}_i^q+
\varphi_{\overline{k}s}\varphi_{\overline{p}q}\overline{F}_{\overline{j}}^kF_j^s
\overline{F}_{\overline{i}}^pF_i^q\bigg)\, .
\end{multline}
To proceed we shall compute the second term of~(\ref{eq5}). The only term in (\ref{eq5}) that does
not
involve second order derivatives of $\varphi$ is
\begin{equation}\label {eq7}
-\sum_{i,j=1}^N\sum_{k,s=1}^M\left(\varphi_kF_{i\overline{j}}^k+\varphi_{\overline{k}}\overline{F}_{i\overline{j}}^k\right)
\left(\varphi_sF_{j\overline{i}}^s+\varphi_{\overline{s}}\overline{F}_{j\overline{i}}^s\right)\, .
\end{equation}
Using again $\varphi_1(w_1,\dots,w_M)=\operatorname{Re}\left(a_kw_k\right)$ in (\ref{eq7}), and
therefore also
(\ref{eq5}), we can simplify (\ref{eq6}) to
\begin{multline*}
\sigma_{2,N}(\varphi_1\circ F)=-\frac
14\sum_{i,j=1}^N\left(a_kF_{i\overline{j}}^k+\overline{a}_k\overline{F}_{i\overline{j}}^k\right)
\left (a_kF_{j\overline{i}}^k+\overline{a}_k\overline{F}_{j\overline{i}}^k\right)\\
=-\frac 14\sum_{i,j=1}^N \left
|a_kF_{i\overline{j}}^k+\overline{a}_k\overline{F}_{i\overline{j}}^k\right|^2\geq 0\, .\\
\end{multline*}
This clearly implies that
\begin{equation}\label{eq8}
F_{i\overline{j}}^k=0 \quad \text{ for all } \ i,j=1,\dots,N, \ k=1,\dots,M.
\end{equation}
Now, due to (\ref{eq8}) and (\ref{eq3}), the second term of (\ref{eq5}) is
equal to
\begin{align*}
&-\sum_{i,j=1}^N\sum_{k,s=1}^M\sum_{p,q=1}^M(\varphi_{ks}\varphi_{\overline{p}\overline{q}}
F_{\overline{j}}^k\overline{F_i}^s\overline{F}_{\overline{i}}^p\overline{F}_{j}^q+
\varphi_{k\overline{s}}\varphi_{p\overline{q}}
F_{\overline{j}}^k\overline{F}_i^sF_{\overline{i}}^p\overline{F_j}^q+\\
&+\varphi_{\overline{k}s}\varphi_{\overline{p}q}
\overline{F}_{\overline{j}}^kF_i^s\overline{F}_{i}^pF_j^q+
\varphi_{\overline{k}\overline{s}}\varphi_{pq}\overline{F}_{\overline{j}}^k\overline{F_i}^s
F_{\overline{i}}^pF_j^q)\, .
\end{align*}
If we take again the function
$\varphi_2(w_1,\dots,w_M)=\operatorname{Re}\left(a_{ks}w_kw_s\right)$, then all terms
with mixed derivatives vanish, and (\ref{eq5}) turns into
\begin{equation*}
\sigma_{2,N}(\varphi_2\circ F)=-\frac 14\sum_{i,j=1}^N\left
|a_{ks}F_{\overline{j}}^kF_i^s+a_{ks}F_{\overline{j}}^sF_i^k+
\overline{a}_{ks}\overline{F}_{\overline{j}}^k\overline{F}_i^s+
\overline{a}_{ks}\overline{F}_{\overline{j}}^s\overline{F}_i^k\right |^2\geq 0\, .
\end{equation*}
Hence,
\begin{equation}\label{eq9}
F_{\overline{j}}^kF_i^s+F_{\overline{j}}^sF_i^k=0, \quad \text{ for all } \ i,j=1,\dots,N, \
k,s=1,\dots,M.
\end{equation}
Now put $k=s$ in (\ref{eq9}). Then we obtain that $F_{\overline{j}}^kF_i^k=0$. Also if we put $i=j$
we get that
$F_{\overline{j}}^kF_j^s+F_{\overline{j}}^sF_j^k=0$. Suppose
that for some $k,j$\ we have that $F_{\overline{j}}^k\neq 0$. Then we
get that $F_i^k=0$, and hence $F_i^s=0$ for all
$i=1,\dots,N$, $s=1,\dots,M$ so $\bar{F}$ is holomorphic.
Otherwise, $F$ is holomorphic, and that finishes the proof.
\end{proof}

As a consequence of Theorem~\ref{holo} we get the following corollary.

\begin{corollary}
Let $1\leq m\leq M$, $1< n\leq N$, and let $\Omega_N\subseteq \mathbb C^N$, $\Omega_M\subseteq
\mathbb C^M$, be connected and open sets. Then the $(M,n)$-subharmonic morphisms are precisely the family of holomorphic and anti-holomorphic mappings defined on
$\Omega_N$ with range in $\Omega_M$.
\end{corollary}

\section{Some remarks on the case $m=2$, and $n=1$}\label{sec_open}

Let $M,N>1$, and let $1\leq m\leq M$, $1\leq n\leq N$. Also, let $\Omega_N\subseteq \mathbb C^N$,
and $\Omega_M\subseteq \mathbb C^M$, be connected and open sets, and assume that
$f=(f_1,\ldots,f_M):\Omega_N\to \Omega_M$ is a holomorphic mapping. In this section we shall
study $(m,n)$-subharmonic morphisms. But on the contrary to Theorem~C where we
assumed that $m\leq n$, we shall here assume that $m=2$, and $n=1$. This case is considerably different,
even for linear holomorphic mappings. We start in Example~\ref{Finally_ex} to construct linear
holomorphic mappings $f,g:\mathbb C^3\to \mathbb C^3$ with the property that $v\circ f\in \mathcal
{SH}_1\left(\mathbb{C}^3\right)$ for every $v\in \mathcal {SH}_2\left(\mathbb{C}^3\right)$, but
there is a function $u\in \mathcal{SH}_2\left(\mathbb{C}^3\right)$ such that $u\circ g\notin
\mathcal {SH}_1\left(\mathbb{C}^3\right)$.

\begin{example}\label{Finally_ex}
Let $f,g:\mathbb C^3\to \mathbb C^3$ be defined by
\[
f(z_1,z_2,z_3)=(z_1,2z_2,3z_3), \ \text { and } \ g(z_1,z_2,z_3)=(z_1,z_2,3z_3)\, .
\]
For
\[
u(z_1,z_2,z_3)=|z_1|^2+|z_2|^2-\frac 12|z_3|^2\in \mathcal {SH}_2\left(\mathbb{C}^3\right)\,
\]
we have that $u\circ g\notin \mathcal {SH}_1\left(\mathbb{C}^3\right)$. Next we shall prove that
for every $v\in \mathcal {SH}_2\left(\mathbb{C}^3\right)$ the function $v\circ f$ is in $\mathcal
{SH}_1\left(\mathbb{C}^3\right)$. Without loss of generality we can assume that $v$ is a smooth
function, since the general case follows by approximation. Let $\mu_1,\mu_2,\mu_3$ be the
eigenvalues of the Hessian matrix $H_v=[a_{jk}]$. Then we have that
\begin{equation}\label{inq}
\begin{aligned}
\sigma_1(\mu_1,\mu_2,\mu_3)&=\mu_1+\mu_2+\mu_3=a_{11}+a_{22}+a_{33}\geq 0,\\
\sigma_{2}(\mu_1,\mu_2,\mu_3)&=\mu_1\mu_2+\mu_2\mu_3+\mu_3\mu_1\\
&=a_{11}a_{22}+a_{22}a_{33}+a_{33}a_{11}-|a_{12}|^2-|a_{23}|^2-|a_{31}|^2\geq 0.
\end{aligned}
\end{equation}
Note that
\[
\operatorname H_{v\circ f}=\left[
  \begin{array}{ccc}
   a_{11} & 2\overline{a_{12}} & 3\overline{a_{13}} \\
   2a_{12} & 4a_{22} & 6\overline{a_{32}} \\
   3a_{31} & 6a_{32} & 9a_{33}\\
  \end{array}
\right],
\]
and therefore, by Example~\ref{ex3} and (\ref{inq}), we get that
\[
\Delta(v\circ f)=a_{11}+4a_{22}+9a_{33}\geq 0\, .
\]\hfill{$\Box$}
\end{example}

We shall need Lemma~\ref{ex4} to be able to deduce Theorem~\ref{cor1}.

\begin{lemma}\label{ex4}
Let $M\geq 2$, $N\geq 1$, and  $t=\min (M,N)$, $a_1,\ldots,a_N\geq 0$. Then the function
$f:\mathbb{R}^M\to\mathbb{R}$
defined by
\[
f(x_1,\dots,x_M)=a_1x_1+\ldots+a_tx_t
\]
has a global minimum that is equal to $0$ on the set
\[
\Omega=\left \{(x_1,\dots x_M)\in \mathbb R^M: x_1+\ldots +x_M\geq 0, \ \ \sum_{j<k}^Mx_jx_k\geq
0\right\}
\]
if, and only if,
\[
\begin{cases}
  \left(\sum_{j=1}^M a_j\right)^2=(M-1)\sum_{j=1}^M a_j^2 \ \text { or }\  a_1=\ldots=a_N& \text{
  if } M\leq N;\\[2mm]

  a_1=\ldots=a_N  & \text{ if } M=N+1;\\[2mm]

  a_1=\ldots=a_N=0  & \text{ if } M>N+1.
 \end{cases}
\]
\end{lemma}
\begin{proof} First, note that since $f$ is linear it will attain its minimum on the boundary of
$\Omega$. Then if,
\[
x_1+\ldots +x_M=0, \ \text{ and } \ \sum_{j<k}^Mx_jx_k\geq 0\, ,
\]
then we have that $x_1=\ldots=x_M=0$. Now consider the case when minimum is attained at
$(0,\dots,0)$. In the case when $N=M$ using the methods of Lagrange multipliers one can check that
function $f(x_1,\dots,x_M)=a_1x_1+\ldots+a_tx_M$ attains its minimum on the set $x_1+\ldots
+x_M\geq 0$ (and therefore also on the set $\Omega$) at the point $(0,\dots,0)$ if, and only if,
$a_1=\ldots=a_M$. In the case when $N\neq M$, we have that $a_1=\ldots=a_M=0$.

Now we shall consider the case when
\begin{equation}\label{1}
x_1+\ldots +x_M\geq 0, \ \text{ and } \ \sum_{j<k}^Mx_jx_k=0\,.
\end{equation}

\emph{Case $M=N$.} To find extremal points of $f$ we shall proceed in a standard manner using the
methods of Lagrange multipliers.   Therefore, we shall solve the following system of equations:
\begin{equation}\label{2}
\begin{aligned}
a_j-\lambda\sum_{k\neq j}x_k&=0, \quad j=1,\ldots,N;\\
\sum_{j<k}x_jx_k&=0\, .
\end{aligned}
\end{equation}
First note that if $\lambda=0$, then $a_1=\ldots=a_N=0$. Therefore, we assume that $\lambda>0$.
Then since
\[
0=\sum_{j<k}x_jx_k=\frac 12\sum_{j=1}^N x_j\left(\sum_{k\neq j}x_k\right)=\frac
{1}{2\lambda}\sum_{j=1}^Na_jx_j\, ,
\]
we get that if the minimum exist it must be $0$. Now, it easy to check that for $A=a_1+\ldots+a_N$
\[
x_j=\lambda^{-1}\left(-a_j+\frac {1}{N-1}A\right)
\]
are the solutions to the first $N$ equation of (\ref{2}). We have to check when the above solution
satisfies conditions (\ref{1}).
We have that
\[
x_1+\ldots+x_N=\lambda^{-1}\frac {1}{N-1}A\geq 0,
\]
since we assumed that  $\lambda>0$. Next, we have that
\begin{multline*}
0=\lambda^2\sum_{j<k}x_jx_k=\sum_{j<k}\left(-a_j+\frac {1}{N-1}A\right)\left(-a_k+\frac
{1}{N-1}A\right)\\
=\sum_{j<k}a_ja_k-\frac {A}{N-1}\sum_{j<k}^N(a_j+a_k)+\frac {N(N-1)}{2(N-1)^2}A^2.
\end{multline*}
This is equivalent to
\[
2(N-1)\sum_{j<k}a_ja_k=(N-2)A^2
\]
or
\begin{equation}\label{3}
\left(\sum_{j=1}^Na_j\right)^2=(N-1)\sum_{j=1}^Na_j^2\, .
\end{equation}
Finally, we can say that the minimum of $f$ exists and it is equal to $0$ if, and only if,
condition~(\ref{3}) is satisfied.

\emph{Case $M<N$.} From Case $M=N$ it follows that the minimum of $f$ exists, and it is equal to
$0$ if, and only if, the following condition is  satisfied
\begin{equation}\label{4}
\left(\sum_{j=1}^Ma_j\right)^2=(M-1)\sum_{j=1}^Ma_j\, .
\end{equation}

\emph{Case $M>N$.} If we take $a_{N+1}=\ldots=a_M=0$, and by using Case $M=N$ we obtain that
\[
\left(\sum_{j=1}^Ma_j\right)^2=(M-1)\sum_{j=1}^Ma_j^2\, .
\]
Hence,
\begin{equation}\label{5}
\left(\sum_{j=1}^Na_j\right)^2=(M-1)\sum_{j=1}^Na_j^2.
\end{equation}
It now follows from (\ref{5}) that
\[
(M-1)\sum_{j=1}^Na_j^2=\left(\sum_{j=1}^Na_j\right)^2\leq N\sum_{j=1}^Na_j^2\, ,
\]
which implies that:
\[
\begin{cases}
  a_1=\ldots=a_N  & \text{ if } M=N+1;\\[2mm]
  a_1=\ldots=a_N=0  & \text{ if } M>N+1.
 \end{cases}
\]
\end{proof}

Now using Lemma~\ref{ex4}, and the proof of Theorem~C we can deduce Theorem~\ref{cor1}.

\begin{theorem}\label{cor1}
Set $M>2$, $N>1$, and let $\Omega_N\subseteq \mathbb C^N$, and $\Omega_M\subseteq \mathbb C^M$, be
connected and open sets, and assume that $f=(f_1,\ldots,f_M):\Omega_N\to \Omega_M$ is a holomorphic
mapping. Also, let $r$ denote the rank of the $N\times M$ matrix
\[
[\nabla_{\mathbb C}f_1\cdots \nabla_{\mathbb C}f_M]\, ,
\]
and let $0\leq s_r(z)\leq s_{r-1}(z)\leq \cdots\leq s_1(z)$ be its singular values at the point
$z\in\Omega_N$. If $f$ is $(2,1)$-subharmonic morphism, then following holds:
\begin{enumerate}\itemsep2mm
\item if $M>N$, then $f$ is constant;

\item if $M\leq N$, then for each $z\in \Omega$ we have that $s_1(z)=\ldots=s_r(z)$ or they must
    satisfy the following condition:
\begin{equation}\label{cor1_1}
\left(\sum_{j=1}^r\sqrt{s_{j}(z)}\right)^2=(r-1)\sum_{j=1}^r s_{j}(z)\, .
\end{equation}
\end{enumerate}
\end{theorem}

\begin{remark}
Let us say a few words about case (2) and~(\ref{cor1_1}). Let $S(z)$ be the $N\times M$ matrix
whose $(i,j)$-entry is $s_i(z)$ if $i=j\leq r$ and $0$ otherwise. Then by the singular value
decomposition theorem there exists an $N \times N$ unitary matrix $V(z)$, and $M \times M$ unitary
matrix $W(z)$ such that
\[
[\nabla_{\mathbb C}f_1\cdots \nabla_{\mathbb C}f_M]=V(z)\times S(z)\times W^*(z)\, .
\]
We conjecture that if the singular values in case (2) satisfy~(\ref{cor1_1}), then $f$ must
linear.
\end{remark}

\end{document}